\documentclass[draft]{amsart}

\usepackage{amsmath, amsthm, amssymb, amstext, amsfonts}
\usepackage{enumerate}
\usepackage{graphicx}
\usepackage[T1]{fontenc}
\usepackage[applemac]{inputenc}

\theoremstyle{plain}

\newtheorem{thm}{Theorem}[section]
\newtheorem{lem}[thm]{Lemma}

\newtheorem{prop}[thm]{Proposition}

\newtheorem*{prob*}{Problem}

\theoremstyle{definition}

\newcommand{\inv}{\ensuremath{^{-1}}}

\newcommand{\rand}{\partial}

\newcommand{\sub}{\subseteq}

\def\ta{tree amalgamation}
\def\qt{quasi-tran\-si\-tive}
\def\qi{quasi-iso\-metric}
\def\qiy{quasi-iso\-me\-try}
\def\qis{quasi-iso\-me\-tries}

\newcommand{\comment}[1]{}

\newcommand{\nat}{{\mathbb N}}

\newenvironment{txteq*}
  {
    \begin{equation*}
    \begin{minipage}[c]{0.85\textwidth} 
    \em                                
  }
  {\end{minipage}\end{equation*}\ignorespacesafterend}

\begin{document}

\title{Tree amalgamations and hyperbolic boundaries}
\author{Matthias Hamann}
\thanks{Supported by the Heisenberg-Programme of the Deutsche Forschungsgemeinschaft (DFG Grant HA 8257/1-1).}
\address{Matthias Hamann, Mathematics Institute, University of Warwick, Coventry, UK}
\date{}

\begin{abstract}
We look at \ta s of locally finite \qt\ hyperbolic graphs and prove that the homeomorphism type of the hyperbolic boundary of such a \ta\ only depends on the homeomorphism types of the hyperbolic boundaries of their factors.
Additionally, we show that two locally finite \qt\ hyperbolic graphs have homeomorphic hyperbolic boundaries if and only if the homeomorphism types of the hyperbolic boundaries of the factors of their terminal factorisations coincide.
\end{abstract}

\maketitle

\section{Introduction}\label{sec_Intro}

Tree amalgamations offer a way to construct new graphs out of existing ones similar as new groups can be constructed via free products with amalgamation or HNN-extensions.
(We refer to Section~\ref{sec_ta} for the definition of \ta s.)
In order to investigate geometric properties of multi-ended \qt\ graphs, it is therefore interesting to see how such properties behave with respect to \ta s.
In this paper, we are looking at the interaction of \ta s with hyperbolicity.
Hyperbolic groups, graphs or spaces play an important role since Gromov's paper~\cite{gromov}.
A first observation in~\cite{HLMR} is the following.

\begin{thm}\label{thm_HMLR7.10}\cite[Theorem 7.10]{HLMR}
Two connected locally finite \qt\ \linebreak graphs are hyperbolic if and only if any of their \ta s of bounded adhesion is hyperbolic.\qed
\end{thm}

Hyperbolic graphs are equipped with a natural boundary, the hyperbolic boundary.
Our first main result says essentially that in a \ta\ of hyperbolic graphs changing the factors without changing the homeomorphism types of their hyperbolic boundary still leads to a \ta\ whose hyperbolic boundary is homeomorphic to the original one.
For this, a \emph{factorisation} of a \qt\ graph $G$ is a tuple $(G_1,\ldots,G_n)$ such that $G$ is obtained by iterated \ta s of all the graphs $G_i$.

\begin{thm}\label{thm_MS1.1}
Let $(G_1,\ldots,G_n)$, $(H_1,\ldots, H_m)$ be factorisations of infinitely-ended \qt\ locally finite hyperbolic graphs $G,H$, respectively, such that the set of homeomorphism types of the hyperbolic boundaries of the factors in $(G_1,\ldots,G_n)$ are the same as those for $(H_1,\ldots, H_m)$.
Then the hyperbolic boundaries $\rand G$ and $\rand H$ are homeomorphic.
\end{thm}

The obvious question that arises is what can be said about the reverse implication of Theorem~\ref{thm_MS1.1}?
While it is false in general, we will prove that it holds if we ask it for \emph{terminal factorisations}.
These are factorisations $(G_1,\ldots,G_n)$, where each $G_i$ has at most one end.
In~\cite{H-Accessibility} it was shown that \qt\ locally finite hyperbolic graphs are accessible in the sense of Thomassen and Woess~\cite{ThomassenWoess}.
Thus, they have a terminal factorisation by \cite[Theorem 6.3]{HLMR}.

\begin{thm}\label{thm_MS1.3}
	Let $(G_1,\ldots,G_n)$, $(H_1,\ldots, H_m)$ be terminal factorisations of in\-fi\-nite\-ly-ended \qt\ locally finite hyperbolic graphs $G,H$, respectively.
	Then the hyperbolic boundaries $\rand G$ and $\rand H$ are homeomorphic if and only if the set of homeomorphism types of the hyperbolic boundaries of the factors in $(G_1,\ldots,G_n)$ are the same as those for $(H_1,\ldots, H_m)$.
\end{thm}

Martin and \'Swi\k{a}tkowski~\cite{MS-HyperbolicGroupsWithHomeoBound} proved group theoretic versions of Theorems~\ref{thm_MS1.1} and~\ref{thm_MS1.3}.
However, our results do not follow from theirs since it is not known whether every locally finite hyperbolic \qt\ graph is \qi\ to some hyperbolic group.

The question whether every locally finite hyperbolic \qt\ graph is \qi\ to some hyperbolic group is a special case of Woess' problem \cite[Problem 1]{Woess-Topo} whether every locally finite transitive graph is \qi\ to some locally finite Cayley graph.
While his problem was settled in the negative by Eskin et al.~\cite{EFW-DLnotTransitive}, their counterexamples, the Diestel-Leader graphs, are not hyperbolic and neither is another counterexample by Dunwoody~\cite{D-AnInaccessibleGraph}.

In Section~\ref{sec_hyp} we define and discuss hyperbolicity and in Section~\ref{sec_ta} \ta s. In both section, we also state the main preliminary results we need for our main results, which we will prove in Section~\ref{sec_pf}.

\section{Preliminaries}

In this section, we state the main definitions and preliminary results that we need for the proofs of Theorems~\ref{thm_MS1.1} and~\ref{thm_MS1.3}.
First, we state some general definitions and then we look at hyperbolic graphs and \ta s more closely in Sections~\ref{sec_hyp} and~\ref{sec_ta}, respectively.

Let $G$ be a graph.
A \emph{ray} is a one-way infinite path and a \emph{double ray} is a two-way infinite path.
Two rays are \emph{equivalent} if for every finite vertex set $S\sub V(G)$ both rays have all but finitely many vertices in the same component of $G-S$.
This is an equivalence relation whose equivalence classes are the \emph{ends} of~$G$.

We call $G$ \emph{transitive} if its automorphism group acts transitively on its vertex set and \emph{\qt} if its automorphism group acts with only finitely many orbits on~$V(G)$.

Let $G$ and $H$ be graphs.
A map $\varphi\colon V(G)\to V(H)$ is a \emph{\qiy} if there are constants $\gamma\geq 1$, $c\geq 0$ such that 
\[
\gamma\inv d_G(u,v)-c\leq d_H(\varphi(u),\varphi(v))\leq\gamma d_G(u,v)+c
\]
for all $u,v\in V(G)$ and such that $\sup\{d_H(v,\varphi(V(G)))\mid v\in V(H)\}\leq c$.
We then say that $G$ is \emph{quasi-isometric} to~$H$.

A finite or infinite path $P$ is \emph{geodesic} if $d_P(x,y)=d_G(x,y)$ for all vertices $x,y$ on~$P$.
It is a \emph{$(\gamma,c)$-quasi-geodesic} if it is the $(\gamma,c)$-\qi\ image of a subpath of a geodesic double ray.

\subsection{Hyperbolic graphs}\label{sec_hyp}

In this section, we will give the definitions and state the lemmas regarding hyperbolic graphs that we need for our results.
For a detailed introduction to hyperbolic graphs, we refer to~\cite{CoornDelPapa,GhHaSur,gromov}.

Let $G$ be a graph and $\delta\geq 0$.
If for all $x,y,z\in V(G)$ and all shortest paths $P_1,P_2,P_3$, one between every two of those vertices, every vertex of~$P_1$ has distance at most $\delta$ to some vertex on either $P_2$ or~$P_3$ then $G$ is \emph{$\delta$-hyperbolic}.
We call $G$ \emph{hyperbolic} if it is $\delta$-hyperbolic for some $\delta\geq 0$.

Two geodesic rays in a hyperbolic graph $G$ are \emph{equivalent} if there is some $M\in\nat$ such that on each ray there are infinitely many vertices of distance at most~$M$ to the other ray.
This is an equivalence relation for hyperbolic graphs whose equivalence classes are the \emph{hyperbolic boundary points} of~$G$.
By $\rand G$ we denote the \emph{hyperbolic boundary} of~$G$, i.\,e.\ the set of hyperbolic boundary points of~$G$, and we set $\widehat{G}:=G\cup\rand G$.

For locally finite hyperbolic graphs, it is possible to equip $\widehat{G}$ with a topology such that $\widehat{G}$ is compact and every geodesic ray converges to the hyperbolic boundary point it is contained in, see~\cite[Proposition 7.2.9]{GhHaSur}.
For us, it suffices to define convergence of vertex sequences to hyperbolic boundary points.
Let $o\in V(G)$.
Let $(x_i)_{i\in\nat}$ be a sequence in $V(G)$.
It \emph{converges} to $\eta\in\rand G$ if for some geodesic ray $y_1y_2\ldots$ and some sequence $(n_i)_{i\in\nat}$ that goes to $\infty$ any geodesic path from $x_i$ to~$y_i$ has distance at least $n_i$ to~$o$.

Since the rays in a tree are always geodesic and two rays that are equivalent regarding the definition of ends eventually coincide, these rays are also equivalent with respect to the definition of the hyperbolic boundary.
Thus, there is a canonical one-to-one correspondence between the ends of trees and their hyperbolic boundary.
The following lemma follows easily.

\begin{lem}\label{lem_treeBound}
The hyperbolic boundaries of locally finite trees are totally disconnected sets.\qed
\end{lem}

By its definition, the hyperbolic boundary is a refinement of the end space.
But even more can be said about this relation, cf.\ e.\,g.\ \cite[Section 7]{KB-BoundaryHypGroup}:

\begin{lem}\label{lem_ConCompBoundAreEnds}
The connected components of the hyperbolic boundary of every locally finite hyperbolic graph correspond canonically to the ends of that graph.\qed
\end{lem}

We are interested in homeomorphism types of hyperbolic boundaries.
That is why the following result is important for us.

\begin{lem}\label{lem_QIinduceHomeo}\cite[III.H Theorem 3.9]{BridsonHaefliger}
Quasi-isometries between locally finite hyperbolic graphs induce canonical homeomorphisms between their boundaries.\qed
\end{lem}

The following lemma is a direct consequence of Woess~\cite[Corollary 5]{W-FixedSets}.

\begin{lem}\label{lem_OneEndBoundInfinite}
The hyperbolic boundary of every one-ended \qt\ locally finite hyperbolic graph is infinite.\qed
\end{lem}

\subsection{Tree-amalgamations}\label{sec_ta}

In this section, we state the definition of \ta s and cite several results about them.

Let $p_1,p_2\in\nat\cup\{\infty\}$.
A tree is \emph{semiregular} or \emph{$(p_1,p_2)$-semiregular} if all vertices in~$V_1$ have the same degree~$p_1$ and all vertices in~$V_2$ have the same degree~$p_2$, where $V_1,V_2$ is the canonical bipartition of its vertex set.

Let $G_1$ and $G_2$ be two graphs and let $T$ be a $(p_1,p_2)$-semiregular tree with canonical bipartition $V_1,V_2$ of its vertex set.
Let $\{S^i_k\mid 0\leq k< p_i\}$ be a set of subsets of~$V(G_i)$ such that all $S^i_k$ have the same cardinality and let $\varphi_{k,\ell}\colon S_k^1\to S_\ell^2$ be a bijection.
Let
\[c\colon E(T)\to\{(k,\ell)\mid 0\leq k< p_1,0\leq\ell< p_2 \}\]
such that for all $v\in V_i$ the $i$-th coordinates of the elements of $\{c(e)\mid v\in e\}$ exhaust the set $\{k\mid 0\leq k< p_i\}$.

For every $v\in V_i$ with $i=1,2$, let $G_i^v$ be a copy of~$G_i$ and let $S_k^v$ be the copy of $S^i_k$ in $G_i^v$.
Let $H:=G_1+G_2$ be the graph obtained from the disjoint union of all graphs $G_i^v$ by adding an edge between all $x\in S_k^v$ and $\varphi_{k,\ell}(x)\in S_\ell^u$ for every edge $vu\in E(T)$ with $v\in V_1$ and $c(vu)=(k,\ell)$.
Let $G$ be the graph obtained from $H$ by contracting all new edges $x\varphi_{k,\ell}(x)$, i.\,e.\ all edges outside of the graphs $G_i^v$.
We call $G$ the \emph{\ta} of $G_1$ and $G_2$ over~$T$ (with respect to the sets $S_k$ and the maps $\varphi_{k,\ell}$) and we denote it by $G_1\ast_T G_2$.
If the \emph{amalgamation tree} $T$ is clear from the context, we simply write $G_1\ast G_2$.
The sets $S_k$ and their copies in~$G$ are the \emph{adhesion sets} of the tree amalgamation.
If the supremum of the diameter of all adhesion sets of a tree amalgamation is finite, then this tree amalgamation has \emph{bounded adhesion}.
Let $\psi\colon V(H)\to V(G)$ be such that every $x\in V(H)$ is mapped to the vertex of~$G$ it ends up after all contractions.
A \ta\ $G_1\ast G_2$ is \emph{trivial} if there is some $G_i^v$ such that the restriction of~$\psi$ to~$G^i_v$ is a bijection $G_i^v\to G_1\ast G_2$.
So a \ta\ of finite adhesion is trivial if $V(G_i)$ is the only adhesion set of~$G_i$ and $p_i=1$ for some $i\in\{1,2\}$.

The \emph{identification size} of a vertex $x\in V(G)$ is the smallest size of subtrees $T'$ of~$T$ such that $x$ is obtained by contracting only edges between vertices in $\bigcup_{u\in V(T')}V(G_j^u)$.
The \ta\ has \emph{finite identification} if the supremum of all identification sizes is finite.

In \cite[Lemma 7.7]{HLMR}, it was shown that for a \ta\ $G\ast H$ of bounded adhesion if $R$ is a geodesic ray in~$G$, then its image in some $G^u$ in $G\ast H$ is a quasi-geodesic ray.
We state a strengthened version of that result for the case that the \ta\ is of adhesion~$1$.
The proof of Lemma~\ref{lem_HMLR7.7+} is contained in the proof of \cite[Lemma 7.7]{HLMR}.

\begin{lem}\label{lem_HMLR7.7+}
	Let $G,H$ be connected locally finite graphs and $G\ast H$ a \ta\ of adhesion~$1$. Then every geodesic in~$G$ is a geodesic in $G\ast H$.\qed
\end{lem}

All following results in this section deal with the interplay of \ta s with \qis\ and are proved in~\cite{H-QuasiIsometriesAndTA}.

\begin{lem}\label{lem_astQI+}\cite[Remark 2.1]{H-QuasiIsometriesAndTA}
Let $G$ and~$H$ be locally finite graphs and $G\ast H$ be a \ta\ of finite identification and bounded adhesion.
Then $G\ast H$ is \qi\ to $G+ H$.\qed	
\end{lem}

The following lemma enables us to change the factors in a \ta\ a bit while staying \qi\ to the original \ta\ but in the result we have more control over the adhesion sets and identification sizes.

\begin{lem}\label{lem_QIandTA3.1}\cite[Lemma 3.1]{H-QuasiIsometriesAndTA}
Let $G$ be a locally finite connected \qt\ graph and let $(G_1,G_2)$ be a factorisation of~$G$.
Then there is a locally finite connected \qt\ graph $H$ that has a factorisation $(H_1,H_2)$ such that the following hold.
\begin{enumerate}[{\rm (1)}]
\item\label{itm_AdhesionG} $G$ is \qi\ to~$H$;
\item\label{itm_AdhesionGi} $G_i$ is \qi\ to~$H_i$ for $i=1,2$;
\item\label{itm_AdhesionOne} $H_1\ast H_2$ has adhesion~$1$;
\item\label{itm_AdhesionDisjoint} all adhesion sets of $H_1\ast H_2$ are distinct;
\item\label{itm_AdhesionsCover} the adhesion sets of~$H_i$ cover $H_i$ for $i=1,2$.\qed
\end{enumerate}
\end{lem}

\begin{lem}\label{lem_PartsFinite}\cite[Lemma 2.9]{H-QuasiIsometriesAndTA}
A connected locally finite \qt\ graph that has a terminal factorisation of only finite graphs is \qi\ to a $3$-regular tree.\qed
\end{lem}

The following theorem plays a central role in our proofs and can be seen as an analogue of Theorem~\ref{thm_MS1.1} for \qis\ of graphs instead of homeomorphisms of hyperbolic boundaries.

\begin{thm}\label{thm_QI1.4}\cite[Theorem 1.4]{H-QuasiIsometriesAndTA}
Let $G$ and $H$ be locally finite quasi-transitive graphs with infinitely many ends and let $(G_1,\ldots,G_n),(H_1,\ldots,H_m)$ be factorisations of $G,H$, respectively.
If $(G_1,\ldots,G_n)$ and $(H_1,\ldots,H_m)$ have the same set of quasi-isometry types of infinite factors, then $G$ and $H$ are quasi-isometric.\qed
\end{thm}

\section{Proofs of the main theorems}\label{sec_pf}

In this section we will prove our main results.
But we need some preliminary results first.

The following lemma is a special case of a result of Steiner and Steiner \cite[Theorem 4]{StSt-GraphClosures}. It is also possible adapt the proof of Martin and \'Swi\k{a}tkowski \cite[Lemma 4.2]{MS-HyperbolicGroupsWithHomeoBound} to our situation to obtain that lemma.

\begin{lem}\label{lem_MS4.2}
Let $G,H$ be locally finite \qt\ hyperbolic graphs and let $f\colon\rand G\to\rand H$ be a homeomorphism.
Then $f$ extends to a homeomorphism ${\widehat{G}\to\widehat{H}}$.\qed
\end{lem}

The next lemma describes the hyperbolic boundary of a \ta\ in terms of its factors and the involved tree.

\begin{lem}\label{lem_canonicalBoundMap}
Let $G$ and $H$ be locally finite hyperbolic \qt\ graphs and let $T$ be a semiregular tree with canonical bipartition $\{U,V\}$ of its vertex set.
Then there exists a canonical bijective map
\[
f\colon \rand T\cup\bigcup_{u\in U}\rand G^u\cup\bigcup_{v\in V}\rand H^v\to \rand (G+_T H)
\]
such that the following hold.
\begin{enumerate}[\rm (1)]
\item\label{itm_canonicalBoundMap1} The preimage of each connected component of $\rand(G+_T H)$ is a connected component of an element of
\[
\{\rand T\}\cup\{\rand G^u\mid u\in U\}\cup\{\rand H^v\mid v\in V\};
\]
\item\label{itm_canonicalBoundMap2} every sequence in some $G^u$ or $H^v$ that converges to some boundary point $\eta\in \bigcup_{u\in U}\rand G^u\cup\bigcup_{v\in V}\rand H^v$ converges to $f(\eta)$ in $G+_TH$;
\item\label{itm_canonicalBoundMap3} every sequence $(v_i)_{i\in\nat}$ with $v_i\in G^{t_i}$ or $v_i\in H^{t_i}$ such that $(t_i)_{i\in\nat}$ converges to $\eta\in\rand T$ converges to $f(\eta)$ in $G+_TH$.
\end{enumerate}\end{lem}

\begin{proof}
Since \qis\ preserve hyperbolicity, we may apply Lemmas~\ref{lem_QIandTA3.1} and~\ref{lem_astQI+} to assume that $G\ast_TH$ is a \ta\ of adhesion~$1$ and distinct adhesion sets are disjoint.
Note for this that \qis\ map distinct boundary points to distinct boundary points and distinct connected components of the boundary to distinct connected components.

Let us define a map
\[
f\colon \rand T\cup\bigcup_{u\in U}\rand G^u\cup\bigcup_{v\in V}\rand H^v\to \rand (G+_T H).
\]
Let $u\in U$ and $\eta\in\rand G^u$.
Any geodesic ray in~$\eta$ is a geodesic ray in~$G+_T H$ as well by Lemma~\ref{lem_HMLR7.7+}.
Thus, two geodesic rays in $G\ast_T H$ that lie in~$G^u$ are equivalent in $G+_T H$ and thus lie in the same boundary point~$\mu$.
We set $f(\eta):=\mu$.
Analogously, we define the image of elements of $\rand G^v$ for $v\in V$.

Now we consider a boundary point $\eta$ of~$T$.
Note that since $T$ is a tree, its boundary points are just its ends.
Let $R$ be a ray in~$\eta$.
Since the \ta\ has adhesion~$1$, there is for each edge $uv$ of~$T$ with $u\in U$ and $v\in V$ a unique edge of $G+_T H$ that corresponds to $uv$ in that its incident vertices lies in $G^u$ and $H^v$ and get identified when constructing the \ta.
Since distinct adhesion sets are disjoint, it follows that for a subpath $u_0u_1u_2u_3$ of~$R$ the edges $e_1,e_2,e_3$, where $e_i$ corresponds to the edge $u_{i-1}u_i$, have the properties that they are distinct and $e_2$ separates $e_1$ and~$e_3$.
By joining $e_1$ and~$e_2$ by a shortest path inside $G^{u_1}$ or $H^{u_1}$, we obtain that $R$ defines a geodesic ray and no matter how we choose the shortest paths to connected the edges $e_1,e_2$, the resulting rays are equivalent and thus converge to the same boundary point $\mu$ of $G+_T H$.
We set $f(\eta):=\mu$.

While defining~$f$, we ensured that it is well-defined.
Let us show that $f$ is injective.
If we consider hyperbolic boundary points $\eta,\mu$ of distinct elements of
\[
X:=\{\rand T\}\cup\{\rand G^u\mid u\in U\}\cup\{\rand H^v\mid v\in V\},
\]
then each of $\eta,\mu$ belongs to either a hyperbolic boundary point of~$T$ or the hyperbolic boundary of a vertex of~$T$ and there is an edge of~$T$ separating these hyperbolic boundary points or vertices of~$T$.
The edge of $G+_T H$ corresponding to that edge of~$T$ separates the $f$-image of those hyperbolic boundary points or of the subgraph $G^u$ or $H^v$.
Thus, $f(\eta)$ and $f(\mu)$ lie in distinct ends of $G+_T H$ and hence are distinct.
If $\eta$ and~$\mu$ are distinct but in a common element of~$X$, then they lie in either $\rand G^u$ or $\rand H^v$ for some $u\in U$ or $v\in V$.
But as geodesic paths and rays in $G^u$ or $H^v$ are geodesic paths and rays in $G\ast_T H$, inequivalent rays in $G^u$ or $G^v$ stay inequivalent in $G+_TH$.
Thus, we have $f(\eta)\neq f(\mu)$ in this case, too.

To show that $f$ is surjective, let $\eta\in\rand (G+_T H)$ and let $R$ be a geodesic ray in~$\eta$.
Since the adhesion of $G\ast_T H$ is~$1$, there is either a subgraph $G^u$ or $H^v$ such that $R$ has all but finitely many vertices of that subgraph or not.
If $R$ meets every such subgraph in only finitely many vertices, let $W\sub V(T)$ consist of those vertices $u\in U$ and $v\in V$ for which $R$ meets $G^u$ and $H^v$.
Note that if~$R$ leaves $G^u$ or $H^v$ once, it has to do so through an adhesion set and since it cannot use the same edge again, it cannot enter the subgraph $G^u$ or $H^v$ anymore.
Thus, $W$ defines a ray in~$T$ and it is straight-forwards to see that the hyperbolic boundary point~$\mu$ of~$T$ that contains this ray is mapped to~$\eta$ by~$f$.
If $R$ has infinitely many vertices in a subgraph $G^u$ or $H^v$, say $G^u$, then it has a subray in~$G^u$ by the above argument that if it leaves $G^u$ once, it never reenters $G^u$.
This subray is geodesic in~$G^u$ since the adhesion sets have size~$1$.
It lies in some hyperbolic boundary point $\mu$ of~$G^u$ and we have $f(\mu)=\eta$ by construction.

So far, we constructed a canonical bijective map $f$ that satisfies (\ref{itm_canonicalBoundMap2}) and~(\ref{itm_canonicalBoundMap3}).
It remains to verify (\ref{itm_canonicalBoundMap1}).
For this, we note that the connected components of the hyperbolic boundary correspond to the ends of the graph by Lemma~\ref{lem_ConCompBoundAreEnds}.
Analogously as in the proof that $f$ is injective, it follows that distinct hyperbolic boundary points in the same end of~$G+_TH$ are mapped to hyperbolic boundary points of some $G^u$ or $G^v$ that lie in the same end of that graph.
\end{proof}

The following proposition is the main step towards the proof of Theorem~\ref{thm_MS1.1}.

\begin{prop}\label{prop_MS4.1}
	Let $G_1,G_2,H_1,H_2$ be locally finite infinite hyperbolic \qt\ graphs such that $\rand G_1, \rand G_2$ is homeomorphic to $\rand H_1,\rand H_2$, respectively.
	Let $G_1\ast G_2$ and $H_1\ast H_2$ be \ta s of bounded adhesion and finite identification.
	Then $\rand(G_1\ast G_2)$ and $\rand (H_1\ast H_2)$ are homeomorphic.
\end{prop}

\begin{proof}
According to Lemma~\ref{lem_QIandTA3.1}, we find locally finite connected \qt\ graphs $G_1',G_2',H_1',H_2'$ such that $G_i$, $H_i$ is \qi\ to~$G_i'$, to~$H_i'$ for $i=1,2$, such that $(G_1',G_2')$, $(H_1',H_2')$ are factorisations of graphs~$G$, $H$ that are \qi\ to \ta s $G_1\ast G_2$, $H_1\ast H_2$, respectively, such that $G_1'\ast G_2'$ and $H_1'\ast H_2'$ have adhesion~$1$, all their adhesion sets are disjoint and every vertex lies in some adhesion set.
As hyperbolicity is preserved by \qis, $G_1',G_2',H_1',H_2'$ are hyperbolic.
By Theorem~\ref{thm_HMLR7.10}, $G$ and $H$ are hyperbolic, too.
Since \qi\ hyperbolic graphs have homeomorphic hyperbolic boundaries by Lemma~\ref{lem_QIinduceHomeo}, the hyperbolic boundaries $\rand (G_1\ast G_2)$, $\rand (H_1\ast H_2)$ are homeomorphic to~$\rand G$, $\rand H$, respectively.
Thus, it suffices to prove the assertion under the assumption that the \ta\ is of adhesion~$1$, all adhesion sets are disjoint and the adhesion sets cover all vertices.
By Lemma~\ref{lem_astQI+} the graphs $G_1\ast G_2$, $H_1\ast H_2$ are \qi\ to $G_1+G_2$, to $H_1+H_2$, respectively.
Thus, it suffices to consider those graphs instead of the \ta s.

For $i=1,2$, let $g_i\colon\rand G_i\to\rand H_i$ be a homeomorphism and let $f_i\colon \widehat{G}_i\to\widehat{H}_i$ be a homeomorphism that extends $g_i$, which exists by Lemma~\ref{lem_MS4.2}.
Let $S^G_k$, $S^H_k$ be the adhesion sets in~$G_1$, in~$H_1$ and $T^G_\ell$, $T^H_\ell$ be the adhesion sets in~$G_2$, in~$H_2$, respectively.
Note that the amalgamation trees in both cases are the countably infinitely regular tree, i.\,e.\ we consider the \ta s $G_1\ast_T G_2$ and $H_1\ast_T H_2$.
Let $c_G$, $c_H$ be the labelings of the edges of~$T$ used for the \ta s $G_1\ast_T G_2$, for $H_1\ast_T H_2$, respectively.
We are going to construct a map $f\colon G_1+G_2\to H_1+H_2$ with the following properties, which obviously proves the assertion.
\begin{enumerate}[(1)]
\item\label{itm_MS4.1_1} $f$ is a bijection;
\item\label{itm_MS4.1_2} $f$ induces an automorphism $f_t$ of~$T$;
\item\label{itm_MS4.1_3} $f$ induces homeomorphisms $\widehat{G}_i^u\to \widehat{H}_i^{f_t(u)}$;
\item\label{itm_MS4.1_4} $f$ induces a homeomorphism $g\colon \rand(G_1+G_2)\to\rand(H_1+H_2)$.
\end{enumerate}
The homeomorphisms in~(3) will be closely related to the homeomorphisms~$f_i$.

Let $v_1,v_2,\ldots$ be an enumeration of $V(T)$ such that for every $i\in\nat$ the vertices $v_1,\ldots, v_i$ induce a subtree of~$T$.
Let $f_{i_1}^{v_1}$ be the map $G_{i_1}^{v_1}\to H_{i_1}^{v_1}$ that is induced by~$f_{i_1}$.
For $j>1$, let $u_j$ be the vertex in $G_{i_j}^{v_j}$ that separates $G_{i_j}^{v_j}$ from $G_{i_1}^{v_1}$ and let $t\in\{v_1,\ldots,v_{j-1}\}$ be the unique neighbour of~$v_j$ in that set.
Let $g_{i_j}^{v_j}$ be the map $G_{i_j}^{v_j}\to H_{i_j}^{v_j}$ induced by~$f_i$.

For $j>1$, let $w_j$ be the vertex in $H_{i_j}^{v_j}$ that separates $H_{i_j}^{v_j}$ from $H_{i_1}^{v_1}$.
Let $h_{i_j}^{v_j}$ be $g_{i_j}^{v_j}$ but the images of $u_j$ and $(g_{i_j}^{v_j})\inv(w_j)$ exchanged.
Note that $h_{i_j}^{v_j}$ still induces a homeomorphism $\widehat{G}_{i_j}^{v_j}\to\widehat{H}_{i_j}^{v_j}$ but now the vertex in $G^t_{3-j}$ that is adjacent to~$u_j$ in the graph $G_1+ G_2$ is mapped by $h_{i_{3-j}}^t$ to the neighbour of~$w_j$ in $H^t_{i_{3-j}}$ in the graph $H_1+H_2$.
Hence, the union of all $h_i^v$ for all $v\in V(T)$ defines a map $f\colon G_1+G_2\to H_1+H_2$ that maps vertices of $G_1+G_2$ that gets identified by constructing $G_1\ast G_2$ to those that gets identified by constructing $H_1\ast H_2$.

By construction, (\ref{itm_MS4.1_1})--(\ref{itm_MS4.1_3}) holds.
Since the compactifications of locally finite hyperbolic graphs are Hausdorff, it suffices to prove that $f$ induces a bijective continuous map $g\colon \rand(G_1+G_2)\to\rand(H_1+H_2)$.
For this, we use Lemma~\ref{lem_canonicalBoundMap} to get a map
\[
\varphi\colon \rand T\cup\bigcup_{u\in U}\rand G_1^u\cup\bigcup_{v\in V}\rand G_2^v\to \rand (G_1+_T G_2)
\]
and a map
\[
\psi\colon \rand T\cup\bigcup_{u\in U}\rand H_1^u\cup\bigcup_{v\in V}\rand H_2^v\to \rand (H_1+_T H_2)
\]
both having the properties as in Lemma~\ref{lem_canonicalBoundMap}.
Then $g:=\psi\circ\varphi\inv$ is a bijective map and it follows from the construction of $\varphi$ and $\psi$ that $f$ induces the restriction of~$g$ to those boundary points of $G_1+G_2$ that are are not in the image of~$\rand T$ by $\varphi$ or $\psi$.
It is not hard to see that $g$ is also induced by~$f$ on the remaining boundary points.
In order to show that $g$ is continuous, we show the slightly stronger assertion that $f\cup g$ is continuous.
For this, it suffices to consider a sequence $(x_i)_{i\in\nat}$ in $G_1+ G_2$ that converges to some $\eta\in\rand (G_1+ G_2)$ and show that $(f(x_i))_{i\in\nat}$ converges to $g(\eta)$.

If $\varphi\inv(\eta)\in \rand (G_i^u)$ for some $i\in\{1,2\}$ and $u\in V(T)$, let $(y_i)_{i\in\nat}$ be a sequence in $G_i^u$ such that $x_i=y_i$ if $x_i\in V(G_i^u)$ and such that $y_i$ separates $x_i$ from $G_i^u$ otherwise.
Then $(y_i)_{i\in\nat}$ converges to~$\eta$, too, and $(f(x_i))_{i\in\nat}$ and $(f(y_i))_{i\in\nat}$ converge to the same boundary point of $H_1+ H_2$ by construction.
Since $(f(y_i))_{i\in\nat}$ lies in $H_i^u$, it converges to $g(\eta)$ and hence $(f(x_i))_{i\in\nat}$ converges to $g(\eta)$.

If $\varphi\inv(\eta)$ lies in no $\rand(G_i^u)$, then it lies in $\rand T$.
Let $t_1t_2\ldots$ be a ray in~$T$ that converges to $\varphi\inv(\eta)$.
Let $(y_i)_{i\in\nat}$ be such that $y_i$ separates $x_i$ and~$\eta$ and such that $y_i$ lies in some $G_j^{t_k}$.
Then $(y_i)_{i\in\nat}$ converges to~$\eta$, too.
As is the previous case, $(f(x_i))_{i\in\nat}$ and $(f(y_i))_{i\in\nat}$ converge to the same boundary point of $H_1+ H_2$, which is $g(\eta)$ by construction.
Thus, $f\cup g$ is continuous.
\end{proof}

Now we are able to prove our main theorems.

\begin{proof}[Proof of Theorem~\ref{thm_MS1.1}]
We will prove the assertion by induction on the number of homeomorphism classes of the hyperbolic boundaries $\rand G_i$.
If there are no homeomorphism classes, then all graphs $G_i$ are finite and so are the graphs $H_i$.
By Lemma~\ref{lem_PartsFinite}, $G$ and~$H$ are \qi\ to $3$-regular trees and hence $G$ is \qi\ to~$H$.

Let us now assume that there is at least one homeomorphism class of hyperbolic boundaries~$\rand G_i$.
Let $G_{i_1},\ldots, G_{i_k}$, $H_{j_1},\ldots,H_{j_\ell}$ be representatives of the infinite \qiy\ types of $G_1,\ldots G_n$, of $H_1,\ldots, H_m$, respectively.
By Theorem~\ref{thm_QI1.4}, $G$ is \qi\ to either $G_{i_1}\ast G_{i_1}$, if $k=1$, or $G_{i_1}\ast\ldots\ast G_{i_k}$, if $k>1$, and similarly $H$ is \qi\ to either $H_{j_1}\ast H_{j_1}$ or $H_{j_1}\ast\ldots\ast H_{j_\ell}$.
Note that the homeomorphism types of the $G_{i_p}$'s and $H_{i_q}$'s are the same.
We apply Theorem~\ref{thm_QI1.4} again to duplicate factors such that $G$ is \qi\ to $G':=G_1',\ldots,G_{m'}'$ and $H$ is \qi\ to $H':=H_1',\ldots, H_{n'}'$, where $G_i'$ and $H_i'$ have homeomorphic hyperbolic boundaries.
As \qis\ do not change the homeomorphism type of the hyperbolic boundary by Lemma~\ref{lem_QIinduceHomeo}, $G$ and $G'$ as well as $H$ and $H'$ have homeomorphic hyperbolic boundaries and by Proposition~\ref{prop_MS4.1} the assertion follows by induction as all factors have a hyperbolic boundary and thus are infinite.
\end{proof}

\begin{proof}[Proof of Theorem~\ref{thm_MS1.3}]
If the factors have the same homeomorphism types of hyperbolic boundaries, it follows from Theorem~\ref{thm_MS1.1} that the hyperbolic boundaries of $G$ and~$H$ are homeomorphic.

Let us now assume that $\rand G$ and $\rand H$ are homeomorphic.
By Lemmas~\ref{lem_treeBound}, \ref{lem_ConCompBoundAreEnds}, \ref{lem_OneEndBoundInfinite} and \ref{lem_canonicalBoundMap}, the non-singular connected components of~$G$ are the hyperbolic boundaries of the $G_i$'s and the non-singular connected components of $\rand H$ are the hyperbolic boundaries of the $H_i$'s.
Thus, the homeomorphism types of $\{\rand G_1,\ldots, \rand G_n\}$ are those of $\{H_1,\ldots,\rand H_m\}$.
\end{proof}

\section*{Acknowledgement}

I thank M.~Pitz for a discussion about~\cite{StSt-GraphClosures}.

\providecommand{\bysame}{\leavevmode\hbox to3em{\hrulefill}\thinspace}
\providecommand{\MR}{\relax\ifhmode\unskip\space\fi MR }
\providecommand{\MRhref}[2]{%
	\href{http://www.ams.org/mathscinet-getitem?mr=#1}{#2}
}
\providecommand{\href}[2]{#2}

\end{document}